\documentclass[letterpaper, 10 pt, conference]{ieeeconf}  
\IEEEoverridecommandlockouts                              

\usepackage{graphicx}
\usepackage{textcomp}
\usepackage[varg,bigdelims]{newtxmath}
\let\labelindent\relax
\usepackage{enumitem}
\usepackage{dsfont, url}
\usepackage{mathtools}
\usepackage[usenames, dvipsnames]{xcolor}
\usepackage[%
colorlinks={true},%
citecolor={Blue},%
urlcolor={PineGreen},%
linkcolor={Sepia}%
]{hyperref}
\usepackage{float}
\usepackage{newfloat}
\usepackage[skip=1pt]{caption}
\usepackage{amsmath} 
\usepackage{amssymb}  
\usepackage{algorithm}
\usepackage[noend]{algpseudocode}
\usepackage{todonotes}
\usepackage{pgf,tikz}
\usepackage{adjustbox}
\usepackage{verbatim} 
\usepackage{xargs}
\usepackage{cite}

\usetikzlibrary{plotmarks}
\usetikzlibrary{chains}
\usepackage{pgfplots}
\pgfplotsset{compat=newest}
\pgfplotsset{plot coordinates/math parser=false}
\newlength\figureheight
\newlength\figurewidth
\usetikzlibrary{shapes,arrows}
\usetikzlibrary{positioning, calc, fit}
\usetikzlibrary{decorations,decorations.pathmorphing,decorations.pathreplacing}
\usetikzlibrary{shapes.misc}
\usetikzlibrary{shadows,trees}

\DeclareMathOperator*{\sbjto}{subject\ to}

\DeclareMathOperator{\sat}{sat}
\DeclareMathOperator{\trace}{tr}

\renewcommand{\leq}{\leqslant}

\renewcommand{\geq}{\geqslant}

\newcommand{\R}{\mathds{R}}
\newcommand{\Nz}{\mathds{N}_0}
\newcommand{\N}{\mathds{Z}_{+}}

\newcommand{\bmat}[1]{\begin{bmatrix}#1\end{bmatrix}}

\newcommand{\norm}[1]{\left\|#1\right\|}
\newcommand{\secref}[1]{Section \ref{#1}}

\newcommand{\transp}{^\top}

\newcommand{\zeros}{\mathbf{0}}

\newcommand{\st}{x}

\newcommand{\control}{u}
\newcommand{\controlset}{\mathds{U}}

\newcommand{\costps}{c_{\mathrm{s}}}
\newcommand{\costfinal}{c_{\mathrm{f}}}

\newcommand{\authority}{u_{\max}}

\newcommand{\Let}{\coloneqq}

\newtheorem{assumption}{Assumption}
\newtheorem{theorem}{Theorem}
\newtheorem{lemma}{Lemma}
\newtheorem{proposition}{Proposition}

\newtheorem{remark}{Remark}

\newtheorem{example}{Example}

\newtheorem{pstatement}{\sc Problem Statement}

\allowdisplaybreaks

\title{\LARGE \bf
	Centralized model predictive control with distributed adaptation
}

\author{Prabhat K. Mishra, Tixian Wang, Mattia Gazzola, Girish Chowdhary 
	\thanks{We gratefully acknowledge financial support from ONR MURI N00014-19-1-2373 and joint CPS USDA grant 2018-67007-28379.}%
	\thanks{ P. K. Mishra, T. Wang and G. Chowdhary are with Coordinated Science Laboratory,
		University of Illinois at Urbana Champaign (UIUC), 
		USA.
		\tt\{pmishra,tixianw2, girishc\}@illinois.edu}%
	\thanks{ M. Gazzola is with the Department of Mechanical Science and Engineering, and National Center for Supercomputing Applications,
		University of Illinois at Urbana Champaign (UIUC), 
		USA.
		\tt\{mgazzola@illinois.edu}%
	\thanks{The first author is thankful to Utsav Sadana and Prof.\ Prashant Mehta for helpful discussions. }
}

\usepackage[normalem]{ulem}

\begin{document}

	\maketitle
	\thispagestyle{empty}
	\pagestyle{empty}

\begin{abstract}  
	A centralized model predictive controller (MPC), which is unaware of local uncertainties, for an affine discrete time nonlinear system is presented. The local uncertainties are assumed to be matched, bounded and structured. In order to encounter disturbances and to improve performance, an adaptive control mechanism is employed locally. The proposed approach ensures input-to-state stability of closed-loop states and convergence to the equilibrium point. Moreover, uncertainties are learnt in terms of the given feature basis by using adaptive control mechanism. In addition, hard constraints on state and control are satisfied.                      
\end{abstract}

\section{Introduction}
Multi-agent systems are often characterized by individual entities that have only partial access to sensory and environmental data, and that have limited communication or computational capabilities, rendering the enactment of a successful centralized, global goal-directed strategy challenging \cite{multiagent, multiagent_MPC}. This is a recurrent scenario shared by a number of systems and applications. For instance, in marine search operations, small unmanned submarines cannot access global positioning system (GPS) while underwater. And although they might be able to exchange information with a fully equipped surface vessel for centralized coordination, their typically limited bandwidth communication may impair operation effectiveness \cite{marine_search}. In agricultural robotics also, ground robots operate in crop fields where they receive noisy GPS data, and while they can communicate with a centralized computer station outside the field, they typically do so with limited and sometimes unreliable bandwidth \cite{agbots_Chowdhary, agriculture}.
\par Similar conditions may be encountered in biological systems as well. Cephalopods, and in particular octopuses, have recently received significant attention in this regard. Indeed, octopuses exhibit the remarkable ability to coordinate, through a relatively limited amount of computing power, virtually infinite degrees of freedom, across eight arms and throughout their compliant, distributed bodies, giving rise to highly complex behaviors \cite{octopus_inspired, octopus_mind}. Although the algorithmic nature of their embodied control system is unknown, it might be reasonable to think of it in terms of generic commands issued by the brain, that are then adapted by the individual arms' neural infrastructure based on local conditions. This is an agile strategy in which the brain does not need to know everything and does not have to detail every low level instruction. Instead, it can more simply `draft' commands that are locally refined by the arms (agents) themselves. As a consequence, no component of the system needs complete information, thus limiting the need for computing resources.
\par Inspired by this paradigm, in this article we present a framework in which a centralized controller responsible for goal-directed global coordination and constraint satisfaction is enriched with distributed adaptation mechanisms to deal with locally observed uncertainties.
Our proposed approach employs tube-based model predictive control (MPC) \cite{Mayne_tube_NLMPC} as a centralized controller to satisfy physical constraints and to achieve an overall objective along with local adaptive control mechanism \cite{haddad_adaptive} to deal with distributed uncertainties. The adaptive control mechanism allows the convergence of closed-loop states to an equilibrium point in addition to input-to-state stability (ISS) achieved due to MPC.
\par MPC techniques are well known for their capability to handle practical constraints at the synthesis stage while optimizing some suitably defined objective function. In many cases the predicted behavior of the system based on nominal model and the actual behavior are not identical due to the presence of uncertainties. Therefore, it is desirable to design a robust control strategy that can guarantee the constraint satisfaction for all disturbance realizations. There are broadly two approaches that deal with bounded uncertainties in the MPC framework: min-max MPC \cite{min-max_MPC} and tube-based MPC \cite{Mayne_tube_NLMPC}. The first approach is based on iteratively solving online an optimal control problem, which involves the minimization of some cost subject to satisfaction of constraints for all possible disturbance sequences. In contrast, tube-based MPC is based on iteratively solving online an optimal control problem for nominal dynamics with tightened constraints, and therefore has modest computational requirements. In this approach all trajectories of a given uncertain system lie in a bounded neighborhood (the tube) of the nominal trajectory (center of the tube), which in turn implies ISS of the closed-loop system. Satisfaction of constraints by uncertain systems is ensured by tightening the constraints for the nominal system \cite{Muller_constraint_tightening, Mesbah_constraint_tightening}.    
\par MPC can be employed along with some other suitable control technique \cite{Borrelli_adaptive_MPC, adaptiveMPC_nonlinear, MPC_ISM_matched}. Adaptive control schemes and MPC are employed for linear \cite{Borrelli_adaptive_MPC} and nonlinear \cite{adaptiveMPC_nonlinear} systems to deal with state-dependent uncertainties. For example, min-max MPC is employed to stabilize the system together with an adaptive estimator to predict the support set for uncertainties, resulting in uniform ultimate boundedness of the closed-loop states \cite{adaptiveMPC_nonlinear}. MPC is also used with the integral sliding mode (ISM) control \cite{MPC_ISM_matched}, where ISM deals with local matched uncertainties and MPC functions as a remote controller. This approach proves ISS of the closed-loop system.  
\par Although MPC can be applied to control a continuous process, MPC based controllers can be invoked only at discrete instants of time \cite{MPC_ISM_matched}. Further, it is argued \cite{haddad_adaptive} that the adaptive control theory developed for continuous time dynamical systems cannot be directly applied to discrete time ones. In particular, the weight update law needs to be modified for discrete time dynamical systems. By numerical experiments it is demonstrated that a native discrete adaptive controller outperforms the discretized version of a continuous adaptive controller \cite{discreteL1_adaptive}. To the best of our knowledge a framework in which a discrete adaptive controller interacts with MPC only at discrete instants of time is missing in literature. The planning-learning framework presented in \cite{planning-learning_Chowdhary} for cooperative multi-agent systems is also relevant to the context of the present article. In \cite{planning-learning_Chowdhary}, each agent learns different models and shares some relevant information with the team. Differently from \cite{planning-learning_Chowdhary}, here we consider non-interacting agents with known dynamics, MPC is employed for planning, and adaptive control for learning only the disturbance weights.

The contributions of this article are: \textit{(a)} a framework that deals with nominal system dynamics centrally and disturbances locally; \textit{(b)} the guarantee of convergence of states to an equilibrium point in addition to ISS while respecting the practical constraints. 
\par This article is structured as follows. We present the system setup and problem formulation in \secref{s:problem_setup}. We present important notions of adaptive control and tube-based MPC in \secref{s:adaptive} and \secref{s:MPC}, respectively. Our algorithm and its stability are presented in \secref{s:stability}. We validate our theoretical results by numerical experiments in \secref{s:experiment} and conclude in \secref{s:epilogue}. Our proofs are given in the appendix.

\begin{figure}
	\centering
	\begin{adjustbox}{width = 0.9\columnwidth}
		\begin{tikzpicture}
		\tikzstyle{blockgreen} = [draw, fill=green!10, rectangle, 
		minimum height=2em, minimum width=1cm] 
		\tikzstyle{blockred} = [draw, fill=red!10, rectangle, 
		minimum height=2em, minimum width=1cm] 			
		\node[coordinate] (0) at (0,0) {};
		\node [blockred, right= 2.5cm of 0] (centralized_MPC) {Centralized MPC};	
		\node [blockgreen, below= 2cm of 0] (left_agent) {First adaptive agent};
		\node [blockgreen, right= 4.5cm of left_agent] (right_agent) {Second adaptive agent};

		\draw[thick ,->] (left_agent.north) -- ($(left_agent.north) + (0,2)$) to node[auto, swap]{$\control_t^{a(1)}, \st_t^{(1)}$}  (centralized_MPC.west);
		\draw[thick ,->] (centralized_MPC.south) to node[auto]{$\control_t^{m(1)}$} (left_agent.east);
		\draw[thick ,->] (right_agent.north) -- ($(right_agent.north) + (0,2)$) to node[auto]{$\control_t^{a(2)}, \st_t^{(2)}$}  (centralized_MPC.east);
		\draw[thick ,->] (centralized_MPC.south) to node[auto, swap]{$\control_t^{m(2)}$} (right_agent.west);
		\end{tikzpicture}
	\end{adjustbox}	
	\caption{Two adaptive agents are shown along with a centralized MPC controller. The states of individual agents and the adaptive control components are transmitted to the centralized MPC controller. The MPC control components are transmitted to the individual adaptive agents.}		
	\label{Fig:Setup}
\end{figure}
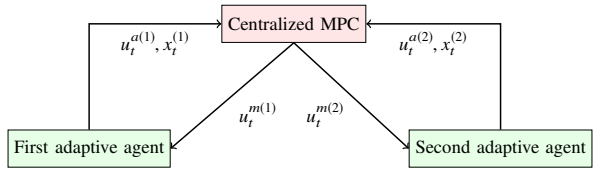  
\section{Problem setup}\label{s:problem_setup}
We let $\R$ denote the set of real numbers, $\Nz$ the set of non-negative integers and $\N$ the set of positive integers. For a given vector $v$ and positive semi-definite matrix $M \succeq \zeros$, $\norm{v}_M^2$ is used to denote $v \transp M v$. For a given matrix $A$, the trace, the largest eigenvalue, pseudo-inverse and Frobenius norm are denoted by $\trace(A)$, $\lambda_{\max}(A)$, $A^{\dagger}$ and $\norm{A}_F$, respectively. By notation $\norm{A}$ and $\norm{A}_{\infty}$, we mean the standard $2-$norm and $\infty-$norm, respectively, when $A$ is a vector, and induced $2-$norm and $\infty-$norm, respectively, when $A$ is a matrix. A vector or a matrix with all entries $0$ is represented by $\zeros$ and $I$ is an identity matrix of appropriate dimensions.
\par Let us consider the discrete time dynamical system 
\begin{equation}\label{e:system}
\st_{t+1} = f(\st_t) + g(\st_t)\left( \control_t + h(\st_t) \right),
\end{equation}
where
\begin{enumerate}[leftmargin = *, nosep, label=(\eqref{e:system}-\alph*), widest = b]
	\item \label{e:constraints} $\st_t \in \mathcal{X} \subset \R^d$, $ \control_t \in \controlset \Let \{v \in \R^m \mid \norm{v}_{\infty} \leq \authority \}$,
	\item $f:\R^d \rightarrow \R^d $, $g:\R^d \rightarrow \R^m$ are given Lipschitz continuous functions with Lipschitz constant $L_f$ and $L_g$, respectively, 
	\item $h(\st_t)$ is the state dependent matched structured uncertainty at time $t$ such that $g(\st_t)h(\st_t) \in \mathds{W} \Let \{v \in \R^d \mid \norm{v}_{\infty} \leq w_{\max} \}$ for every $\st_t \in \R^d$. We assume that a feature basis function $\phi: \R^p \rightarrow \R^q$ is given such that $h(\st_t) = W\phi(\st_t)$ where $W \in \R^{m\times q}$ is an unknown matrix.
	\item We further assume that there exist $\delta_g, \delta_{\phi} >0$ such that $\norm{g(x)} \leq \delta_g$ and $\norm{\phi(x)} \leq \delta_{\phi}$ for each $x \in \mathcal{X}$. 
	\item We have enough control authority, so that $\authority > 2 \norm{W}_F\delta_{\phi}$.	 
  \end{enumerate}  
  \begin{remark}
Let us consider a system of $M$ agents in which each agent is characterized by a different dynamics and different structure of matched uncertainties. For instance, the agent $i$ has the dynamics of the following form: 
\begin{equation}\label{e:system_overall}
\st_{t+1}^{(i)} = f^{(i)}(\st_t^{(i)}) + g^{(i)}(\st_t^{(i)})\left( \control_t^{(i)} + h^{(i)}(\st_t^{(i)}) \right),
\end{equation} 
where $f^{(i)}, g^{(i)}$ and $h^{(i)}$ can be defined as in \eqref{e:system} for each $i = 1, \cdots, M$. We assume that there is no direct interaction among agents, and a centralized controller has a copy of the dynamics of the individual sub-systems but is unaware of the uncertainties associated with individual agents. Therefore, it is reasonable to equip each agent with some adaptation technique to counter the locally observed disturbances.  
  \end{remark}
  \begin{pstatement}
 Present a stabilizing control framework for \eqref{e:system}, which can be generalized to non-interacting, multi-agent systems \eqref{e:system_overall}, respects physical constraints \ref{e:constraints}, optimizes a given performance index, and learns uncertainties in terms of given feature basis functions for each agent.   
  \end{pstatement}
 \par Our proposed solution is based on constraint satisfaction and cost minimization capabilities of MPC, and disturbance rejection capability of the adaptive control. Since agents are non-interacting and an overall objective needs to be achieved, we consider MPC as a centralized controller. Then, for each agent $i$, the control $\control_t^{(i)}$ at time $t$ is defined as:
 \begin{equation} 
 \control_t^{(i)} = \control_t^{a(i)} + \control_t^{m(i)}, 
 \label{e:cuntrol_u}
 \end{equation}
 where $\control_t^{a(i)}$  is the adaptive component and $\control_t^{m(i)}$ is the MPC component.
An illustration of this architecture for two adaptive agents is provided in Fig.~\ref{Fig:Setup}. Each agent knows its feature basis function and computes its adaptive control by using the weight update law defined in Section \ref{s:adaptive}. Each agent transmits its state information and the adaptive control component to the centralized controller, and the centralized controller transmits the corresponding MPC component to each agent. In rest of the manuscript, we consider only one agent for notational simplicity. We revisit the example of two agents again in the numerical experiment section. 
 \par Following (\ref{e:cuntrol_u}), the control $\control_t$ for a single agent is
 \begin{equation}\label{e:total_conrol} 
 \control_t = \control_t^a + \control_t^m, 
 \end{equation}
 where $\control_t^a$ and $\control_t^m$ are the adaptive and MPC components. The centralized MPC controller employs only the nominal dynamics of \eqref{e:system} which is given below for easy reference
\begin{equation}\label{e:nominal}
\st_{t+1} = f(\st_t) + g(\st_t) \control_t^m  \Let \bar{f}(\st_t, \control_t^m).
\end{equation}
Therefore, the dynamics \eqref{e:system} can be written as
\begin{equation}\label{e:actual_agent}
\st_{t+1} = \bar{f}(\st_t, \control_t^m) + g(\st_t) \left( \control_t^a + h(\st_t)\right).
\end{equation}
The centralized controller generates reference state and control trajectories off-line for (each) agent by using the nominal dynamics \eqref{e:nominal}. To ensure that actual constraints can be satisfied by (each) agent  \eqref{e:actual_agent}, the central controller appropriately tightens the constraints (\secref{s:MPC}). These reference trajectories are used by the centralized MPC to generate control sequences for \eqref{e:actual_agent} so that the state and control sequences stay within tubes around the nominal reference trajectory. The MPC components computed by the centralized controller are transmitted to (each) agent at every time step.
\par In a broad sense, the MPC component $\control_t^m$ is responsible for ISS of closed-loop states in the presence of bounded disturbances, and the adaptive control component $\control_t^a$ is responsible for disturbance rejection by adapting to the disturbance weight $W$. 
\section{Adaptive agent}\label{s:adaptive}
In the present article, we adopt the weight update law of \cite{haddad_adaptive} to learn the weight $W$ with help of an adaptive controller $\control_t^a = K_t \phi(\st_t)$. The dynamics \eqref{e:nominal} can be written as
\begin{equation}\label{e:nominal_adaptive}
\st_{t+1} = \bar{f}(\st_t, \control_t^m) + g(\st_t)\left( K_t + W \right) \phi(\st_t). 
\end{equation}
For some $\varepsilon > 0 $, $\Gamma = \Gamma\transp \succ 0$, $\lambda_{\max}(\Gamma) < 2$, the following weight update law adapted from \cite{haddad_adaptive} is employed here: 
\begin{equation}\label{e:weight_update}
K_{t+1} = K_t - \frac{\Gamma g(\st_t)^{\dagger}(\st_{t+1} - \bar{f}(\st_t,\control_t^m)) \phi(\st_t)\transp}{ \varepsilon + \norm{\phi(\st_t)}^2},
\end{equation}
when $g(\st_t) \neq 0$, otherwise $K_{t+1} = K_t$.
Let us define $\tilde{\control}_t \Let (K_t + W)\phi(\st_t)$ and $\tilde{K}_t \Let K_t + W$, then by adding $W$ to both sides of \eqref{e:weight_update} the weight update law can be written as
\begin{equation}\label{e:weight_update_ana}
\tilde{K}_{t+1} = \tilde{K}_t - \frac{\Gamma \tilde{\control}_t \phi(\st_t)\transp}{\varepsilon + \norm{\phi(\st_t)}^2}.
\end{equation}
Since $\tilde{\control}_t$ is not known, \eqref{e:weight_update_ana} is used only for analysis. We have the following result:
\begin{lemma}\label{lem:convergence_error}
	Let us consider the dynamics \eqref{e:nominal_adaptive} and the weight update law \eqref{e:weight_update}, then $\norm{g(\st_t)\tilde{\control}_t} \rightarrow 0$ as $t \rightarrow \infty$.
\end{lemma}
\par A proof of Lemma \ref{lem:convergence_error} is given in the appendix. Further, the adaptive control component is shown to be bounded in the following Lemma:
\begin{lemma}\label{lem:bound_adaptive_control}
	The adaptive component of control $\control_t^a$ is bounded by $\authority^a \Let \left(\sqrt{\frac{\lambda_{\max}(\Gamma)}{\lambda_{\min}(\Gamma)}} + 1 \right)\norm{W}_F\delta_{\phi}$ for all $t$. Further,
	$
	\norm{\tilde{K}_t}_F \leq \sqrt{\frac{\lambda_{\max}(\Gamma)}{\lambda_{\min}(\Gamma)}}\norm{W}_F
	$ for all $t$.
\end{lemma}
\par A proof of Lemma \ref{lem:bound_adaptive_control} is given in the appendix.

\begin{remark}\label{rem:apparent_disturbance}
	It is immediate to note from Lemma \ref{lem:bound_adaptive_control} that \[\norm{g(\st_t)\tilde{\control}_t} \leq \delta_g\sqrt{\frac{\lambda_{\max}(\Gamma)}{\lambda_{\min}(\Gamma)}}\norm{W}_F\delta_{\phi} \Let w_{\max}^{\prime} \text{ for all } t. \]
\end{remark}

\section{Centralized controller}\label{s:MPC}
In this section we provide a summary and important results of tube-based MPC relevant to the context of the present article. The discussion closely follows \cite{Mayne_tube_NLMPC} with differences related to the adaptive controller discussed in \secref{s:adaptive}. We note that since a control part is used by the adaptive controller, full control authority is not available to MPC. In addition, although the disturbance in dynamical system \eqref{e:system} at time $t$ is $g(x_t)h(x_t)$, the disturbance observed by MPC is $g(\st_t)\tilde{\control}_t$ \eqref{e:nominal_adaptive}. We clarify these distinctions in the following. 
\subsection{Offline reference governor}\label{s:reference_governor}
Let us fix an optimization horizon $N \in \N$, and let $Q\succeq \zeros$ and $R \succ \zeros$ be given positive semi-definite and positive definite matrices, respectively. In view of Lemma \ref{lem:bound_adaptive_control}, we fix $\Gamma$ such that $\authority > \authority^a$ and define a set 
$\controlset^{\prime} \Let \{v \in \R^m \mid \norm{v}_{\infty} \leq \authority - \authority^a \}.$   
The offline reference governor consists of the following optimization problem without terminal cost but with a terminal constraint:
\begin{equation}\label{e:reference_governor}
\begin{aligned}
\min_{(\control_i^r)_{i=0}^{N-1}} & \quad \sum_{i=0}^{N-1} \norm{\st_i^r}^2_Q + \norm{\control_i^r}_R^2  \\
\sbjto \quad  & \st_0^r = \st_0, \st_N^r = \zeros, \\
&\st_{i+1}^r = \bar{f}(\st_i^r, \control_i^r), \st_i^r \in \mathcal{X}_r \subset \mathcal{X},  \\
& \control_i^r \in \controlset_r \subset \controlset^{\prime}; i = 0, \ldots, N-1.  
\end{aligned}
\end{equation}
The tightened constraints $\st_i^r \in \mathcal{X}_r$ and $\control_i^r \in \controlset_r\Let \left\{ v \in \R^m \mid \norm{v}_{\infty} \leq \authority^r \right\}$ are used in the reference governor design for a nominal dynamics so that the actual constraints can be satisfied for uncertain system \eqref{e:system}. We refer readers to \cite{Muller_constraint_tightening, Mesbah_constraint_tightening, convex_enclosures} for recent results on constraint tightening. In this article, we follow the approach of \cite[Section 7]{Mayne_tube_NLMPC} for the determination of $\mathcal{X}_r$ and $\controlset_r$. We can define reference sequences for state and control as follows:
\begin{align}
(\st_t^r)_{t \in \Nz} &\Let \{ (\st_t^r)_{t = 0}^N, 0, \ldots \}, \label{e:reference_signal_state}\\
(\control_t^r)_{t \in \Nz} &\Let \{ (\control_t^r)_{t = 0}^{N-1}, 0, \ldots \} \label{e:reference_signal_control}.
\end{align} 
\subsection{Online reference tracking}
In this section, we consider the state and control reference trajectories defined in \eqref{e:reference_signal_state} and \eqref{e:reference_signal_control}, respectively, for nominal system, and present a reference tracking MPC for actual system. Let 
\[ \costps(\st_{t+i \mid t}, \control_{t+i \mid t}) \Let \norm{\st_{t+i \mid t} - \st_{t+i}^r}_Q^2 + \norm{\control_{t+i \mid t} - \control_{t+i}^r}_R^2\]
be the cost per stage at time $t+i$ predicted at time $t$. Let $\costfinal^{\prime}(\cdot)$ be a local control Lyapunov function for the nominal system, defined by $\costfinal^{\prime}(x) \Let x \transp Q_f x$ for some $Q_f \succ 0$. Let $\mathcal{X}_f \Let \{x \mid  \costfinal^{\prime}(x) \leq \alpha \}$ for some $\alpha > 0$ be a terminal set. In the tube-based MPC approach, the terminal cost functional is defined such that $\costfinal(\cdot) = \beta \costfinal^{\prime}(\cdot)$ for some $\beta \geq 1$. 
We have the following assumption:
\begin{assumption}\label{as:stability}
	\rm{
		There exists a control $\control^{\prime} \in \controlset^{\prime} $ such that the following holds
		\begin{equation}
		\costfinal \left( \bar{f}(x, \control^{\prime}) \right) + \costps(x, \control^{\prime}) -\costfinal(x)  \leq 0 \text{ for every } x \in \mathcal{X}_f.
		\end{equation}		
	}
\end{assumption}
We define the following cost
\begin{equation}
V\left( \st_t, (\control_{t+i \mid t})_{i=0}^{N-1} \right) \Let \costfinal(\st_{t+N \mid t}) + \sum_{i=0}^{N-1}\costps(\st_{t+i \mid t}, \control_{t+i \mid t})
\end{equation}
and the following optimization problem is solved iteratively online:
\begin{equation}\label{e:MPC}
\begin{aligned}
V_m(\st_t) \Let & \min_{(\control_{t+i \mid t})_{i=0}^{N-1}}  V(\st_t, (\control_{t+i \mid t})_{i=0}^{N-1})  \\
\sbjto \quad & \st_{t \mid t} = \st_t \\
& \st_{t+ i +1 \mid t} = \bar{f}(\st_{t+i \mid t}, \control_{t+i\mid t})   \\
& \control_{t+i \mid t} + \control_t^a \in \controlset; \quad i = 0, \cdots, N-1. 
\end{aligned}
\end{equation}
Note that the above optimal control problem does not involve terminal constraints. The terminal constraint set $\mathcal{X}_f$ is defined only for the purpose of analysis.
Let $X^c(\st_t^r)$ be a level set of radius $c$ around $\st_t^r$ at time $t$, defined by $X^c(\st_t^r) \Let \{x \mid V_m(\st_t) \leq c\} $. For any $c > 0$, $\beta \geq \frac{c}{\alpha}$, for all $\st_{t} \in X^c(\st_t^r)$ implies the terminal state $\st_{t+N \mid t} \in \mathcal{X}_f$ \cite[Proposition 3.16]{Mayne_tube_NLMPC}. A tube is defined as a sequence of level sets
\begin{equation}
X_c \Let \{ X^c(\st_0^r), X^c(\st_1^r), X^c(\st_2^r), \ldots \}.
\end{equation} 
The optimal value of \eqref{e:MPC} satisfies the following property: 
\begin{lemma}\label{lem:candidate_lyapunov}
	There exist $c_1, c_2 > 0$ such that 
	\begin{equation}
	c_1 \norm{\st_t - \st_t^r}^2 \leq V_m(\st_t) \leq c_2 \norm{\st_t - \st_t^r}^2,
	\end{equation}
	for every $\st_t \in \R^d$ and $t \in \Nz$.
\end{lemma}  
\par A proof of Lemma \ref{lem:candidate_lyapunov} is given in the appendix. Let us define a set 
\[ \mathds{W}^{\prime} \Let \{v \in \R^d \mid \norm{v} \leq w_{\max}^{\prime} \}\]
where $w_{\max}^{\prime}$ is defined in Remark \ref{rem:apparent_disturbance}. Since the disturbance observed by MPC belongs to the set $\mathds{W}^{\prime}$, we recast the following results for completeness.
\begin{lemma}{\cite[Proposition 3]{Mayne_tube_NLMPC}}\label{lem:lipschitz_value}
	There exists $c_3>0$ such that $\st_t \mapsto V_m(\st_t)$ is Lipschitz continuous with Lipschitz constant $c_3$, when $\st_t \in X_c(\st_t^r)+ \mathds{W}^{\prime}$ for every $t$.
\end{lemma}
\begin{lemma}{\cite[Proposition 2]{Mayne_tube_NLMPC}}\label{lem:bound_on_predicted_next}
If Assumption \ref{as:stability} is satisfied, then for every $\st_t \in X_c+\mathds{W}^{\prime}$ the following holds
	\begin{equation}\label{e:bound_on_predicted_next}
	V_m(\st_{t+1 \mid t}) - V_m(\st_t) \leq -\costps(\st_t, \control_{t \mid t}) \text{ for all } t.
	\end{equation}
\end{lemma}
\section{Algorithm and stability}\label{s:stability} 
In this section, we present our algorithm and provide our main result on stability. 
\begin{algorithm} [H]
	\caption{Centralized MPC with distributed adaptation}
	\label{a:algo}
	\begin{algorithmic}[1]
		\Require  $\st_0, \phi$
		\State choose $\Gamma$ 
		\State initialize $K_0 = \zeros$, $t=0$ 
		\For{each $t$}		
		\State \label{al:first_step}compute $\control_t^a = K_t \phi(\st_t)$ 
		\State solve \eqref{e:MPC}, set $\control_t^m = \control_{t \mid t}$ 
		\State apply $\control_t = \control_t^m + \control_t^a$ to the system and measure $\st_{t+1}$
		\State compute $K_{t+1}$ by the weight update law \eqref{e:weight_update}
        \EndFor
	\end{algorithmic}
\end{algorithm}
\par We have the following main results:
\begin{theorem}\label{th:stability}
Consider the dynamical system \eqref{e:system} and let Assumption \ref{as:stability} hold. Then under the control computed by the Algorithm \ref{a:algo} the closed-loop system \eqref{e:system} is ISS and $\norm{\st_t} \rightarrow 0$ as $t \rightarrow \infty$.	
\end{theorem}
\begin{proof} 
	We compute a bound on $\Delta_m(\st_t) = V_m(\st_{t+1})- V_m(\st_t)$ as follows:
	\begin{align}\label{e:ISS}
	\Delta_m(\st_t) & = V_m(\st_{t+1 \mid t})- V_m(\st_t) + V_m(\st_{t+1}) - V_m(\st_{t+1 \mid t}) \notag \\
	& \leq - \costps(\st_t, \control_t^m) + V_m(\st_{t+1}) - V_m(\st_{t+1 \mid t}) \text{ by Lemma \ref{lem:bound_on_predicted_next}} \notag \\
	& \leq - \costps(\st_t, \control_t^m) + c_3\norm{\st_{t+1}-\st_{t+1\mid t}} \quad \text{ by Lemma \ref{lem:lipschitz_value}} \notag \\
	& = - \costps(\st_t, \control_t) + c_3\norm{g(\st_t)\tilde{\control}_t}.
	\end{align}	
	Since $\norm{g(\st_t)\tilde{\control}_t}\leq w_{\max}^{\prime}$ is bounded for all $t$ by Remark \ref{rem:apparent_disturbance}, the dynamical system \eqref{e:system} is ISS under the control computed by the Algorithm \ref{a:algo}. Further, by Lemma \ref{lem:convergence_error}, $\norm{g(\st_t)\tilde{\control}_t} \rightarrow 0$ as $t \rightarrow \infty$, which implies $\norm{\st_t} \rightarrow 0$ as $t \rightarrow \infty$.
\end{proof}
In our approach the stabilizing controller and the associated Lyapunov-like function do not satisfy conditions $2$ and $3$ of \cite[Theorem 1]{haddad_adaptive}. Therefore, we used a different notion of stability. We can obtain similar results while satisfying state and control constraints at the expense of an extra condition on $g(\cdot)$. 
\begin{proposition}\label{th:main}
	Consider the dynamical system \eqref{e:system} and let Assumption \ref{as:stability} hold. Let $g(\zeros) = \zeros$. Then the control computed by Algorithm \ref{a:algo} guarantees that $(\st_t, K_t) \rightarrow \mathcal{M} \Let \{(x,K) \in \R^d \times \R^{m \times q} \mid x =0, K_{t+1} = K_t \} $ as $t \rightarrow \infty$. 
\end{proposition}
\par A proof of Proposition \ref{th:main} is given in the appendix.

\begin{remark}
Our main result on stability in Theorem \ref{th:stability} is based on Lipschitz continuity of the value function of MPC. Since in the proposed framework MPC considers both the adaptive control component and matched uncertainties as disturbances, we utilized the Key Technical Lemma \cite[Lemma 6.2.1]{goodwin_adaptive_filtering} of adaptive control to show that the disturbance is vanishing with time in Lemma \ref{lem:convergence_error}. This way we are able to achieve not only ISS, but also convergence of states to the equilibrium while satisfying state and control constraints.
\end{remark}
\begin{remark}
Stability result of \cite[Theorem 1]{haddad_adaptive} is based on partial stability via a Lyapunov-like function without using the Key Technical Lemma. In contrast to \cite[Theorem 1]{haddad_adaptive}, we have state and control constraints, and we do not assume the existence of a stabilizing controller and associated Lyapunov-like function. Therefore, we cannot guarantee the satisfaction of conditions $2$ and $3$ of \cite[Theorem 1]{haddad_adaptive} a priori. However, we are able to prove similar stability result by adding an extra term in the candidate Lyapunov function on the expense of an extra condition on $g(\cdot)$. 
\end{remark}

\section{Numerical experiment}\label{s:experiment}
In this section we present two examples to corroborate our theoretical results. 
In both examples, we use MATLAB based software packages CasAdi \cite{casadi} and mpctools \cite{mpctools} for simulations.  
\begin{example}
We consider the model and parameters of stirred tank reactor which was considered as a benchmark model in \cite{Mayne_tube_NLMPC}. The example considers a nonlinear chemical reaction described by product concentration ($y(t)$) and temperature ($z(t))$ with a nonlinear periodic disturbance as follows:
\begin{equation}\label{e:cont-dynamics}
\begin{aligned}
\dot{y}_t & = \theta_1(1-y_t) - \theta_2 y_t e^{\frac{\theta_3}{z_t}} \\
\dot{z}_t & = \theta_1(\theta_4-z_t) + \theta_2 y_t e^{\frac{\theta_3}{z_t}}  -\theta_5(z_t - \theta_6) (\control_t + w_t)
\end{aligned}
\end{equation}
where $\st_t = \bmat{y_t & z_t}\transp$, $\theta_1=0.05, \theta_2=300, \theta_3=-5, \theta_4=0.3947, \theta_5=0.117, \theta_6=0.3816$. In the above model $w_t  = W\sin t$ is a time varying disturbance with $W=2$ is weight, which is unknown to the controller and $\sin t$ is considered as a known basis function. The continuous time dynamic model \eqref{e:cont-dynamics} is discretized by using the sampling interval $0.5$ seconds within the MATLAB script and Algorithm \ref{a:algo} is employed to steer the initial state $\st_0 = \bmat{0.9831 & 0.3918}\transp$ to a locally unstable equilibrium state $\st_e = \bmat{0.2632 & 0.6519}\transp$. The equilibrium control is computed to be $\control_e=0.7583$ as in \cite{Mayne_tube_NLMPC}. The cost functional and the reference trajectories are appropriately modified for the non-zero equilibrium. The state and control sets are 
$
\mathcal{X}  = [0,\ 2]\times [0,\ 2], \text{ and }  \controlset = [0,\ 2]$.
The constraint sets for the offline reference governor are 
$\mathcal{X}_r=\mathcal{X}$, $\controlset_r=[0.02,2]\subset\controlset$ and the terminal state $\st_N^r$ is chosen to be $x_e$. The simulation parameters are chosen to be $N=40, \Gamma = 1.5, \varepsilon = 0.1, Q = 0.5I, R=0.5, Q_f = 10^5I$.
We compare the performance of the tube-based MPC controller \cite{Mayne_tube_NLMPC} and our proposed controller, using the same set of parameters under the same scenario. 
Our numerical results are illustrated in Fig.\ \ref{fig:comparison}, which demonstrate that the tube-based MPC controller produces small oscillations around the equilibrium point. Note that when the same tube-based MPC controller is equipped with the adaptive mechanism from our Algorithm \ref{a:algo}, the closed-loop states asymptotically converge to the equilibrium without violating state and control constraints.  
\begin{figure}[t]
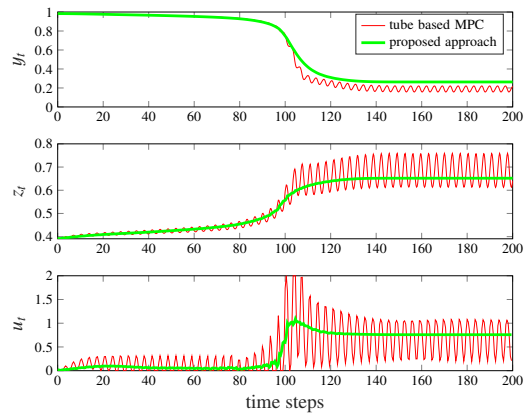

	\centering
	\begin{adjustbox}{width = 0.8\columnwidth}
%
	\end{adjustbox}
\caption[Comparison between tube-based MPC and our proposed approach]{The proposed approach ensures asymptotic convergence of the closed-loop states to the equilibrium point.}
	\label{fig:comparison}
\end{figure}
\end{example}
\begin{example}
In this second example we seek to demonstrate the benefit of local adaptation. For this purpose, we consider two aircrafts subject to wing-rock dynamics, which are adverse lateral-longitudinal dynamics experienced at high angle of attack. Wing-rock is a benchmark dynamical system that has been widely used to evaluate adaptive controllers. The key difference here are two folds, first we simulate the scenario where a common controller is synthesized assuming nominal dynamics, but each aircraft must adapt to its unknown individual operating conditions; second, unlike existing adaptive control results \cite{yucelen2012command,chowdhary2010concurrent,chowdhary2013concurrent} we require that the control and system states remain bounded in a pre-specified constraint set. 
This problem can be posed as that of a centralized controller synthesized with the commonly known nominal dynamics and a local adaptation capability as in Fig. \ref{Fig:Setup}. For both agents, the system matrix pair $(A,B)$ and the feature basis function $\phi(\cdot)$ are the same but the initial conditions and the disturbance weights are different. Letting $\delta_t$ denote the roll angle in radian, and $p_t$ denote the roll rate in radian per second, the state of the wing-rock dynamics model is $x_t \Let \bmat{\delta_t & p_t}\transp$ at time $t$. We discretized the nominal dynamics used in \cite{yucelen2012command} with a sampling interval of $0.05$ seconds and chose the disturbance terms as follows: 
\begin{equation}
x_{t+1}^{(i)} = Ax_t^{(i)} + B\Big(u_t^{(i)} + W^{(i)}\phi(x_t^{(i)})\Big),
\label{eq:disc-sys}
\end{equation}
for $i=1,2$, where $A=\begin{bmatrix}
1 & 0.05 \\ 0 & 1
\end{bmatrix}$, $B=\begin{bmatrix}
0 \\ 0.05
\end{bmatrix}$, $W^{(1)} = 5  \bar{W}$,
$W^{(2)} = 6  \bar{W}$, and \[ \bar{W} = \bmat{0.1414 & 0.5504 & -0.0624 & 0.0095 & 0.0215}.\]
The feature basis function $\phi(\cdot)$ is saturated by a standard saturation function to meet the requirement of bounded disturbances as $\phi(x) = \sat(\beta(x))$, where $\beta(x) = \begin{bmatrix}
x_1 & x_2 & |x_1|x_2 & |x_2|x_1 & x_1^3
\end{bmatrix}\transp $ and $\sat(\cdot)$ is a standard saturation function with the threshold $0.1\max_{x \in \mathcal{X}} \norm{\beta(x)}_{\infty}$. 
\par The goal is to stabilize the wing-rock dynamics by driving states to the origin from arbitrarily chosen initial states $x_0^{(1)}=\bmat{-10 & 6}\transp$ and $x_0^{(2)}=\bmat{10 & -10}\transp$, respectively, where roll angle is represented in degree and roll rate in degree per second. For both agents, sets within which the state and control are to be constrained are given as 
\[ \mathcal{X}  = [-30,\ 30]\times [-15,\ 15], \text{ and }
\controlset  = [-60,\ 60]. \]
\par We use the MATLAB based software package MPT 3.0 \cite{MPT3} to compute the polytopes needed for the constraint tightening. The simulation parameters are chosen to be same for both agents with $N=100$ for \eqref{e:reference_governor} and $N=20$ otherwise. We set  $\Gamma = 1.5, \varepsilon = 1, Q = I, R=0.1$. Fig.\ ~\ref{fig:roll_angle} and Fig.\ ~ \ref{fig:roll_rate} depict the trajectories of the states of the system obtained by the Algorithm \ref{a:algo}, which are compared with the reference trajectory obtained in Section \ref{s:reference_governor}. Figure~\ref{fig:roll_angle} demonstrates that the roll angle of the first agent converges to the reference trajectory when controller is designed by our approach while the MPC based controller, which does not have the adaptive control component, although follows the reference trajectory but has oscillations of comparatively larger amplitude. Since the second agent has larger disturbance weights, oscillations are also bigger than those of the first agent. However, our approach results in comparatively faster convergence for both agents. We have similar observations in Fig.\ \ref{fig:roll_rate} for roll rate. 
\begin{figure}
	\centering
	\begin{adjustbox}{width = 0.8\columnwidth}
		\input{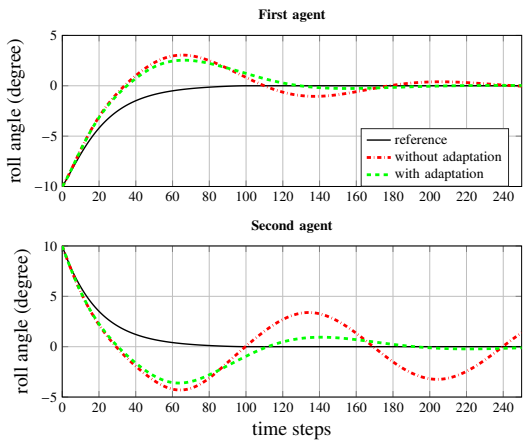}		
	\end{adjustbox}
\caption{Roll angle of each agent converge to the reference and our approach shows faster convergence for both agents.}
\label{fig:roll_angle}
\end{figure}

\begin{figure}
	\centering
	\begin{adjustbox}{width = 0.8\columnwidth}
		\input{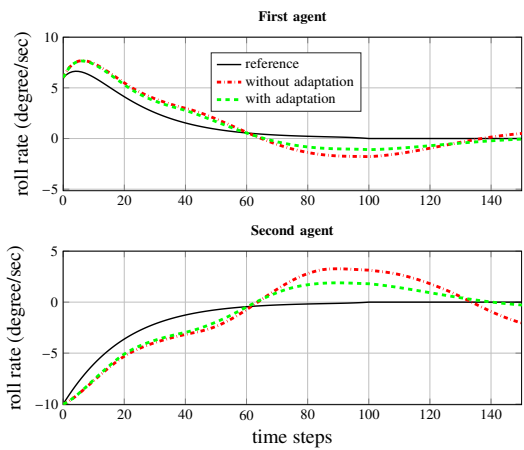}		
	\end{adjustbox}
	\caption{Roll rate of each agent converge to the reference and our approach shows faster convergence for both agents.}
	\label{fig:roll_rate}
\end{figure}
\end{example}
\section{Epilogue}\label{s:epilogue}
We presented an algorithm that combines MPC and adaptive control mechanism, and ensures stability of the closed-loop states. We compared our approach with tube-based MPC. The proposed framework is limited to bounded, matched and structured uncertainties. Some suitable combinations of adaptive control techniques and MPC with required modifications can be employed within the framework of the present approach to relax our assumptions on uncertainties and to achieve a different control objective. The proposed approach may be extended to the settings of unreliable channels \cite{PDQ_intermittent} and/or interacting agents.  

\bibliographystyle{IEEEtran}        
\bibliography{ref} 

\begin{thebibliography}{10}
\providecommand{\url}[1]{#1}
\csname url@samestyle\endcsname
\providecommand{\newblock}{\relax}
\providecommand{\bibinfo}[2]{#2}
\providecommand{\BIBentrySTDinterwordspacing}{\spaceskip=0pt\relax}
\providecommand{\BIBentryALTinterwordstretchfactor}{4}
\providecommand{\BIBentryALTinterwordspacing}{\spaceskip=\fontdimen2\font plus
\BIBentryALTinterwordstretchfactor\fontdimen3\font minus
  \fontdimen4\font\relax}
\providecommand{\BIBforeignlanguage}[2]{{%
\expandafter\ifx\csname l@#1\endcsname\relax
\typeout{** WARNING: IEEEtran.bst: No hyphenation pattern has been}%
\typeout{** loaded for the language `#1'. Using the pattern for}%
\typeout{** the default language instead.}%
\else
\language=\csname l@#1\endcsname
\fi
#2}}
\providecommand{\BIBdecl}{\relax}
\BIBdecl

\bibitem{multiagent}
A.~Dorri, S.~S. Kanhere, and R.~Jurdak, ``Multi-agent systems: A survey,''
  \emph{IEEE Access}, vol.~6, pp. 28\,573--28\,593, 2018.

\bibitem{multiagent_MPC}
R.~R. Negenborn, B.~DeSchutter, and J.~Hellendoorn, ``Multi-agent model
  predictive control: A survey,'' \emph{arXiv preprint arXiv:0908.1076}, 2009.

\bibitem{marine_search}
A.~J. Murphy, M.~J. Landamore, and R.~W. Birmingham, ``The role of autonomous
  underwater vehicles for marine search and rescue operations,''
  \emph{Underwater Technology}, vol.~27, no.~4, pp. 195--205, 2008.

\bibitem{agbots_Chowdhary}
W.~McAllister, D.~Osipychev, A.~Davis, and G.~Chowdhary, ``Agbots: Weeding a
  field with a team of autonomous robots,'' \emph{Computers and Electronics in
  Agriculture}, vol. 163, p. 104827, 2019.

\bibitem{agriculture}
R.~Shamshiri, C.~Weltzien, I.~A. Hameed, I.~Yule, T.~Grift, S.~K. Balasundram,
  L.~Pitonakova, D.~Ahmad, and G.~Chowdhary, ``Research and development in
  agricultural robotics: A perspective of digital farming,''
  \emph{International Journal of Agricultural and Biological Engineering},
  vol.~11, no.~4, pp. 1--14, 2018.

\bibitem{octopus_inspired}
M.~Calisti, F.~Corucci, A.~Arienti, and C.~Laschi, ``Bipedal walking of an
  octopus-inspired robot,'' in \emph{Conf. on Biomimetic and Biohybrid
  Systems}.\hskip 1em plus 0.5em minus 0.4em\relax Springer, 2014, pp. 35--46.

\bibitem{octopus_mind}
J.~Mather, ``What is in an octopus’s mind,'' \emph{Animal Sentience},
  vol.~26, no.~1, 2019.

\bibitem{Mayne_tube_NLMPC}
D.~Q. Mayne, E.~C. Kerrigan, E.~J. Van~W., and P.~Falugi, ``Tube-based robust
  nonlinear model predictive control,'' \emph{International Journal of Robust
  and Nonlinear Control}, vol.~21, no.~11, pp. 1341--1353, 2011.

\bibitem{haddad_adaptive}
T.~Hayakawa, W.~M. Haddad, and A.~Leonessa, ``A {Lyapunov-based} adaptive
  control framework for discrete-time non-linear systems with exogenous
  disturbances,'' \emph{International Journal of control}, vol.~77, no.~3, pp.
  250--263, 2004.

\bibitem{min-max_MPC}
D.~Lim{\'o}n, T.~Alamo, F.~Salas, and E.~F. Camacho, ``Input to state stability
  of min--max {MPC} controllers for nonlinear systems with bounded
  uncertainties,'' \emph{Automatica}, vol.~42, no.~5, pp. 797--803, 2006.

\bibitem{Muller_constraint_tightening}
J.~K{\"o}hler, M.~A. M{\"u}ller, and F.~Allg{\"o}wer, ``A novel constraint
  tightening approach for nonlinear robust model predictive control,'' in
  \emph{American Control Conf.}\hskip 1em plus 0.5em minus 0.4em\relax IEEE,
  2018, pp. 728--734.

\bibitem{Mesbah_constraint_tightening}
T.~L.~M. Santos, A.~D. Bonzanini, T.~A.~N. Heirung, and A.~Mesbah, ``A
  constraint-tightening approach to nonlinear model predictive control with
  chance constraints for stochastic systems,'' in \emph{American Control
  Conf.}\hskip 1em plus 0.5em minus 0.4em\relax IEEE, 2019, pp. 1641--1647.

\bibitem{Borrelli_adaptive_MPC}
M.~Bujarbaruah, S.~H. Nair, and F.~Borrelli, ``A semi-definite programming
  approach to robust adaptive {MPC} under state dependent uncertainty,''
  \emph{arXiv preprint arXiv:1910.04378}, 2019.

\bibitem{adaptiveMPC_nonlinear}
X.~Wang, L.~Yang, Y.~Sun, and K.~Deng, ``Adaptive model predictive control of
  nonlinear systems with state-dependent uncertainties,'' \emph{International
  Journal of Robust and Nonlinear Control}, vol.~27, no.~17, pp. 4138--4153,
  2017.

\bibitem{MPC_ISM_matched}
G.~P. Incremona, A.~Ferrara, and L.~Magni, ``Asynchronous networked {MPC} with
  {ISM} for uncertain nonlinear systems,'' \emph{IEEE Trans. on Auto. Control},
  vol.~62, no.~9, pp. 4305--4317, 2017.

\bibitem{discreteL1_adaptive}
M.~Elnaggar, M.~S. Saad, H.~A. Fattah, and A.~L. Elshafei, ``Discrete time
  $l_1$ adaptive control for systems with time-varying parameters and
  disturbances,'' in \emph{55th Conf. on Decision and Control}.\hskip 1em plus
  0.5em minus 0.4em\relax IEEE, 2016, pp. 2115--2120.

\bibitem{planning-learning_Chowdhary}
N.~K. Ure, G.~Chowdhary, Y.~F. Chen, J.~P. How, and J.~Vian, ``Distributed
  learning for planning under uncertainty problems with heterogeneous teams,''
  \emph{Journal of Intelligent \& Robotic Systems}, vol.~74, no. 1-2, pp.
  529--544, 2014.

\bibitem{convex_enclosures}
M.~E. Villanueva, X.~Feng, R.~Paulen, B.~Chachuat, and B.~Houska, ``Convex
  enclosures for constrained reachability tubes,'' \emph{IFAC-PapersOnLine},
  vol.~52, no.~1, pp. 118--123, 2019.

\bibitem{goodwin_adaptive_filtering}
G.~C. Goodwin and K.~S. Sin, \emph{Adaptive filtering prediction and
  control}.\hskip 1em plus 0.5em minus 0.4em\relax Dover Publications Inc.,
  Mineola, New york, 2009.

\bibitem{casadi}
J.~A.~E. Andersson, J.~Gillis, G.~Horn, J.~B. Rawlings, and M.~Diehl,
  ``{CasADi} -- {A} software framework for nonlinear optimization and optimal
  control,'' \emph{Mathematical Programming Computation}, vol.~11, no.~1, pp.
  1--36, 2019.

\bibitem{mpctools}
\BIBentryALTinterwordspacing
M.~J. Risbeck and J.~B. Rawlings, ``{MPCT}ools: Nonlinear model predictive
  control tools for {CasADi} (octave interface),'' 2016. [Online]. Available:
  \url{https://bitbucket.org/rawlings-group/octave-mpctools}
\BIBentrySTDinterwordspacing

\bibitem{yucelen2012command}
T.~Yucelen and E.~Johnson, ``Command governor-based adaptive control,'' in
  \emph{AIAA Guidance, Navigation, and Control Conf.}, 2012, p. 4618.

\bibitem{chowdhary2010concurrent}
G.~Chowdhary and E.~Johnson, ``Concurrent learning for convergence in adaptive
  control without persistency of excitation,'' in \emph{49th IEEE Conf. on
  Decision and Control}.\hskip 1em plus 0.5em minus 0.4em\relax IEEE, 2010, pp.
  3674--3679.

\bibitem{chowdhary2013concurrent}
G.~Chowdhary, T.~Yucelen, M.~M{\"u}hlegg, and E.~N. Johnson, ``Concurrent
  learning adaptive control of linear systems with exponentially convergent
  bounds,'' \emph{International Journal of Adaptive Control and Signal
  Processing}, vol.~27, no.~4, pp. 280--301, 2013.

\bibitem{MPT3}
M.~Herceg, M.~Kvasnica, C.~N. Jones, and M.~Morari, ``{Multi-Parametric Toolbox
  3.0},'' in \emph{European Control Conf.}, 2013, pp. 502--510.

\bibitem{PDQ_intermittent}
P.~K. Mishra, D.~Chatterjee, and D.~E. Quevedo, ``Stochastic predictive control
  under intermittent observations and unreliable actions,'' \emph{Automatica},
  vol. 118, p. 109012, 2020.

\bibitem{trace_bound}
J.~Snyders, ``On the error matrix in optimal linear filtering of stationary
  processes,'' \emph{IEEE Trans. on Info. Theory}, vol.~19, no.~5, pp.
  593--599, 1973.

\bibitem{lang_analysis}
S.~Lang, \emph{Undergraduate analysis}.\hskip 1em plus 0.5em minus 0.4em\relax
  Springer Science \& Business Media, 1996.

\end{thebibliography}
         
\appendix

\begin{proof}[Lemma \ref{lem:convergence_error}]
Since $0 \leq \norm{g(\st_t)\tilde{\control}_t} \leq \delta_g \norm{\tilde{\control}_t}$, it is sufficient to prove $\norm{\tilde{\control}_t} \rightarrow 0$ as $t \rightarrow \infty$. 	
Let us define $V_a(K_t) \Let \trace(\tilde{K}_t\transp \Gamma^{-1} \tilde{K}_t)$, then from \cite[Lemma 6]{trace_bound}, we get that
\begin{equation}\label{e:adaptive_lyapunov}
 \lambda_{\min}(\Gamma^{-1})\trace(\tilde{K}_t\tilde{K}_t\transp)\leq V_a(K_t) \leq \lambda_{\max}(\Gamma^{-1})\trace(\tilde{K}_t\tilde{K}_t\transp).
\end{equation}
  Let us compute the difference 
 \begin{align}\label{e:adaptive_lyapunov_difference}
	\Delta_a(K_t) &\Let V_a(K_{t+1}) - V_a(K_t) = \trace(\tilde{K}_{t+1}\transp \Gamma^{-1} \tilde{K}_{t+1}) - \trace(\tilde{K}_t\transp \Gamma^{-1} \tilde{K}_t) \notag \\
	& = \frac{\trace \left( (\Gamma \tilde{\control}_t \phi(\st_t)\transp)\transp \tilde{\control}_t \phi(\st_t)\transp \right)}{ (\varepsilon + \norm{\phi(\st_t)}^2)^2} - \frac{2 \trace \left( \tilde{K}_t\transp \tilde{\control}_t \phi(\st_t)\transp \right) }{\varepsilon + \norm{\phi(\st_t)}^2}\notag \\
	& = \frac{\trace(\norm{\phi(\st_t)}^2 \tilde{\control}_t \transp \Gamma \tilde{\control}_t )}{(\varepsilon + \norm{\phi(\st_t)}^2)^2} - \frac{2 \trace(\tilde{\control}_t \tilde{\control}_t\transp)}{\varepsilon + \norm{\phi(\st_t)}^2}\notag \\
	& \leq - \frac{ \tilde{\control}_t \transp (2I - \Gamma) \tilde{\control}_t }{\varepsilon + \norm{\phi(\st_t)}^2}.
\end{align}
Since $\lambda_{\max}(\Gamma) < 2$, $2I - \Gamma \succ 0$. It is clear from \eqref{e:adaptive_lyapunov_difference} that 
\begin{align}\label{e:bound_adaptive_lyapunov}
& V_a(K_T) - V_a(K_0) \leq - \sum_{t=0}^{T-1}\frac{ \tilde{\control}_t \transp (2I - \Gamma) \tilde{\control}_t }{\varepsilon + \norm{\phi(\st_t)}^2} \notag \\
& V_a(K_T)  \leq V_a(K_0) \text{ for all } T, 
\end{align}
which in turn implies that
\begin{align*}
\lim_{T \rightarrow \infty }\sum_{t=0}^{T-1}\frac{ \tilde{\control}_t \transp (2I - \Gamma) \tilde{\control}_t }{\varepsilon + \norm{\phi(\st_t)}^2}   &\leq \lim_{T \rightarrow \infty } \left( V_a(K_0) - V_a(K_T) \right)  \notag \\
\sum_{t=0}^{\infty}\frac{ \tilde{\control}_t \transp (2I - \Gamma) \tilde{\control}_t }{\varepsilon + \norm{\phi(\st_t)}^2}   &\leq V_a(K_0). 
\end{align*}
Therefore, $\frac{ s_t }{(\varepsilon + \norm{\phi(\st_t)}^2)^{1/2}} \in \ell_2$, where $s_t = \left( \tilde{\control}_t \transp (2I - \Gamma) \tilde{\control}_t \right)^{\frac{1}{2}}$, which implies $\frac{ s_t }{(\varepsilon + \norm{\phi(\st_t)}^2)^{1/2}} \rightarrow 0 $. Since $\varepsilon$ is a constant and $\norm{\phi(\cdot)}$ is uniformly bounded, the Key Technical Lemma \cite[Lemma 6.2.1]{goodwin_adaptive_filtering} trivially implies $s_t \rightarrow 0$. Since $0 \leq (\lambda_{\min}(2I-\Gamma))^{1/2}\norm{\tilde{\control}_t} \leq s_t$, \cite[Theorem 2.7]{lang_analysis} ensures $\norm{\tilde{\control}_t} \rightarrow 0$. 
\end{proof}
\begin{proof}[Lemma \ref{lem:bound_adaptive_control}]
It is clear from \eqref{e:bound_adaptive_lyapunov} and \eqref{e:adaptive_lyapunov} that $\lambda_{\min}(\Gamma^{-1})\norm{\tilde{K}_t}_F^2 \leq \lambda_{\max}(\Gamma^{-1})\norm{\tilde{K}_0}_F^2$. Therefore, if $K_0 = \zeros$, then $\tilde{K}_0 = W$, which in turn implies
\begin{equation}
\norm{\tilde{K}_t}_F^2 \leq \frac{\lambda_{\max}(\Gamma)}{\lambda_{\min}(\Gamma)}\norm{W}_F^2.
\end{equation}
Therefore, $\norm{\control_t^a} \leq \norm{\tilde{\control}_t} + \norm{W\phi(\st_t)} 
\leq \norm{\tilde{K}_t}_F\delta_{\phi} + \norm{W}_F\delta_{\phi} \leq \authority^a$.
\end{proof}
\begin{proof}[Lemma \ref{lem:candidate_lyapunov}]
The existence of $c_1 > 0$ can be proved similarly as in \cite[Proposition 2]{Mayne_tube_NLMPC}. In order to prove the existence of $c_2>0$, we utilize Lipscitz continuity of $f$ and $g$. Since $V_m(\st_t) \leq \min_{(\control_{t+i\mid t})_0^{N-1}} V(\st_t, (\control_{t+i\mid t})_0^{N-1})$, we can show that $V_m(\st_t) \leq  V(\st_t, (\control_{t+i}^r)_0^{N-1}) = \st_{t+N\mid t}\transp Q_f \st_{t+N\mid t} + \sum_{i=0}^{N-1} (\st_{t+i \mid t}-\st_{t+i}^r)\transp Q (\st_{t+i \mid t}-\st_{t+i}^r) \leq \lambda_{\max}(Q_f)\norm{\st_{t+N\mid t}}^2 + \sum_{i=0}^{N-1} \lambda_{\max}(Q) \norm{\st_{t+i\mid t} - \st_{t+i}^r}^2 $. Further,
\begin{align*}
\st_{t+i \mid t} - \st_{t+i}^r &= f(\st_{t+i -1 \mid t}) -f(\st_{t+i -1}^r) \\
& \quad + \left( g(\st_{t+i -1 \mid t})-g(\st_{t+i -1}^r)\right) \control_{t+i-1}^r \\
\norm{\st_{t+i \mid t} - \st_{t+i}^r} & \leq \left( L_f+L_g\norm{\control_{t+i-1}^r} \right) \norm{\st_{t+i-1 \mid t} - \st_{t+i-1}^r}  \\
&\leq \bar{L}\norm{\st_{t+i-1 \mid t} - \st_{t+i-1}^r}  \leq \bar{L}^i\norm{\st_{t} - \st_{t}^r},
\end{align*}
where $\bar{L} = L_f+L_g\authority^r$. Since $\st_{t+N}^r = \zeros$ for all $t$, there exists $c_2 = \bar{L}^N\lambda_{\max}(Q_f) + \sum_{i=0}^{N-1}\bar{L}^i\lambda_{\max}(Q)$.
\end{proof}

\begin{proof}[Proposition \ref{th:main}]
Since $\st_N^r = \zeros$ due to \eqref{e:reference_governor}, $\norm{g(\st_{t+N}^r)} = 0$ for all $t$. We begin with the following candidate Lyapunov function to see the stability of our algorithm
\begin{equation}
V(\st_t, K_t) = V_m(\st_t) + c_3 \left(\sum_{i=t}^{t+N-1}\norm{g(\st_i^r)}\norm{\tilde{\control}_i} \right) + a V_a(K_t),
\end{equation}
where $a \geq \frac{(c_3L_g)^2(\varepsilon + \delta_{\phi}^2)}{4\lambda_{\min}(Q)(2-\lambda_{\max}(\Gamma))}$.
We compute the Lyapunov difference $\Delta(\st_t, K_t) \Let V(\st_{t+1}, K_{t+1}) - V(\st_t, K_t) $ as follows
\begin{equation*}
\Delta(\st_t, K_t) = \Delta_m(\st_t) - c_3\norm{g(\st_t^r)}\norm{\tilde{\control}_t} + a \Delta_a(K_t),
\end{equation*}
where $\Delta_m(\st_t)$ is defined in \eqref{e:ISS} and $\Delta_a(K_t)$ in \eqref{e:adaptive_lyapunov_difference}.
Let us first consider $\Delta_m(\st_t) - c_3\norm{g(\st_t^r)}\norm{\tilde{\control}_t}$,
\begin{align*}
& \Delta_m(\st_t) - c_3\norm{g(\st_t^r)}\norm{\tilde{\control}_t} \leq - \costps(\st_t, \control_t^m) + c_3\norm{g(\st_t)\tilde{\control}_t} \notag \\
& \quad - c_3\norm{g(\st_t^r)}\norm{\tilde{\control}_t}, \notag \\
& \quad \leq - \costps(\st_t, \control_t^m) + c_3\norm{g(\st_t)}\norm{\tilde{\control}_t}- c_3\norm{g(\st_t^r)}\norm{\tilde{\control}_t}, \notag \\
& \quad \leq - \costps(\st_t, \control_t^m) + c_3L_g\norm{\st_t-\st_t^r} \norm{\tilde{\control}_t}, \notag \\
& \quad \leq - \costps(\st_t, \control_t^m) + \lambda_{\min}(Q)\norm{\st_t - \st_t^r}^2 +  \frac{(c_3L_g)^2}{4\lambda_{\min}(Q)}\norm{\tilde{\control}_t}^2,
\end{align*}
where the last inequality is due to the Peter-Paul inequality. 
Therefore,
\begin{align}\label{e:Lyapunov_difference}
\Delta(\st_t, K_t) & \leq - \norm{\st_t - \st_t^r}_{(Q-\lambda_{\min}(Q)I)}^2 - \norm{\control_t^m - \control_t^r}_R^2   \notag \\
& \quad -  \frac{ \tilde{\control}_t\transp \left( a(2I - \Gamma) - \frac{(c_3L_g)^2}{4\lambda_{\min}(Q)}(\varepsilon + \norm{\phi(\st_t)}^2)I\right)\tilde{\control}_t }{\varepsilon + \norm{\phi(\st_t)}^2}
\end{align}
Due to the choice of $a$, we have $a(2I - \Gamma) - \frac{(c_3L_g)^2}{4\lambda_{\min}(Q)}(\varepsilon + \norm{\phi(\st_t)}^2)I\succ 0$, which in turn implies $\Delta(\st_t, K_t)\leq 0$ for all $t$. The rest of the result follows from \cite[Theorem 1]{haddad_adaptive}.
\end{proof}
\end{document}